\documentclass[a4paper]{article}

 \usepackage{amsmath, amsfonts}
 \usepackage{amssymb}
 \usepackage{amsthm}
 \usepackage{mathrsfs}
 \usepackage{graphicx}
 \usepackage{color}

\newtheorem{thm}{Theorem}

\newtheorem{lem}[thm]{Lemma}

\newtheorem{conjecture}[thm]{Conjecture}

\title{Tilted Sperner Families}
\author{Imre Leader\thanks{Department of Pure Mathematics and Mathematical Statistics, Centre for Mathematical Sciences, Cambridge CB3 0WB, United Kingdom. E-mail: I.Leader@dpmms.cam.ac.uk} \and Eoin Long\thanks{Department of Pure Mathematics and Mathematical Statistics, Centre for Mathematical Sciences, Cambridge CB3 0WB, United Kingdom. E-mail: E.P.Long@dpmms.cam.ac.uk. Research of the second author is supported by a Benefactor Scholarship from St. John's College, Cambridge.}}
\begin{document}

\maketitle

\begin{abstract}
Let $\cal A$ be a family of subsets of an $n$-set such that $\mathcal A$ does not contain distinct sets
$A$ and $B$ with $\vert A \backslash B \vert = 2 \vert B \backslash A \vert$. How large can $\cal A$
be? Our aim in this note is to determine the maximum size of such an $\mathcal A$. This
answers a question of Kalai. We also give some related results and
conjectures.
\end{abstract}


\section{Introduction}

A set system $\mathcal A\subseteq \mathcal {P}[n]=\mathcal {P}(\{1,\ldots ,n\})$ is said to be an \emph{antichain} or \emph{Sperner family} if $A\not\subset B$ for all distinct $A,B\in \mathcal A$. Sperner's theorem \cite{sper} says that any antichain $\mathcal A$ has size at most 
$\binom {n}{\lfloor n/2\rfloor }$. (See \cite{com} for general background.)

Kalai \cite{kal} noted that the antichain condition may be restated as: $\mathcal A$ does not
contain $A$ and $B$ such that, in the subcube of the $n$-cube spanned by $A$ and $B$,
they are the top and bottom points. He asked what happens if we `tilt' this
condition. For example, suppose that we instead forbid $A$,$B$ such that $A$ is
1/3 of the way up the subcube spanned by $A$ and $B$? Equivalently, $\mathcal A$ cannot
contain two sets $A$ and $B$ with $|A\backslash B|=2|B\backslash A|$.

An obvious example of such a system is any level set $[n]^{(i)}=\{ A\subset [n]:|A|=i\} $. Thus we may
certainly achieve size $\binom {n}{\lfloor n/2\rfloor }$. The system $[n]^{(\lfloor n/2\rfloor )}$ is not maximal, as
we may for example add to it all sets of size $\lfloor n/4\rfloor -1$ -- but that is a rather
small improvement. Kalai \cite{kal} asked if, as for Sperner families, it is still
true that our family $\mathcal A$ must have size $o(2^n)$.

Our aim in this note is to verify this. We show that the middle layer is
asymptotically best, in the sense that the maximum size of such a family
is $(1+o(1)) \binom {n}{\lfloor n/2\rfloor }$. We also find the exact extremal system, for $n$ even
and sufficiently large. We give similar results for any particular
`forbidden ratio' in the subcube spanned.

What happens if, instead of forbidding a particular ratio, we instead
forbid an absolute distance from the bottom point? For example, for
distance 1 this would correspond to the following: our set system $\mathcal A$ must not contain
sets $A$ and $B$ with $|A\backslash B|=1$.  How large can $\mathcal A$ be?

Here the situation is rather different, as for example one cannot take an
entire level. We give a construction that has size about $\frac {1}{n} \binom {n}{\lfloor n/2\rfloor }$,
which is about (a constant fraction of) $1/n^{3/2}$ of the whole cube. But we are not able to show that this
is optimal: the best upper bound that we are able to give is ${2^n}/{n}$.
However, if we strengthen the condition to $\mathcal A$ not having $A$ and $B$ with
$|A\backslash B| \leq 1$ then we are able to show that the greatest family has size $\frac{1}{n} \binom{n}{\lfloor n/2\rfloor }$, up to a multiplicative constant. 
\\


\section{Forbidding a fixed ratio}


In this section we consider the problem of finding the maximum size of a family $\mathcal A$ of subsets of $[n]$ which satisfies $p|A\backslash B|\neq q|B\backslash A|$ for all $A,B\in \mathcal A$ where $p:q$ is a fixed ratio. Initially we will focus on the first non-trivial case $1:2$ (note that $1:1$ is trivial as then the condition just forbids two sets of the same size in $\mathcal A$) and then at the end of the section we extend these results to any given ratio.

As mentioned in the Introduction, for the ratio $1:2$ we actually obtain the extremal family when $n$ is even and sufficiently large. This family, which we will denote by $\mathcal B_0$, is a union of level sets: $\mathcal B_0=\cup _{i\in I}[n]^{(i)}$. Here the set $I$ is defined as follows: $I=\{a_i:i\geq 0\}\cup \{b_i:i\geq 0\} $, where $a_0=b_0=\frac {n}{2}$ and $a_i$ and $b_i$ are defined inductively by taking $a_i=\lceil \frac {a_{i-1}}{2}\rceil -1$ and $b_i=\lfloor \frac {b_{i-1}+n}{2}\rfloor +1$ for all $i$. For example, if $n=2^k$ then $I=\{2^{k-1}\}\cup \{2^i-1:0\leq i\leq k-1\}\cup \{2^k-2^i+1:0\leq i\leq k-1\} $. Noting that for any sets $A$ and $B$ with either (i) $|A|=l$ where $l<\frac {n}{2}$ and $|B|>2l$ or (ii) $|A|=l$ where $l>\frac {n}{2}$ and $|B|<2l-n$ we have $|A\backslash B|\neq 2|B\backslash A|$, we see that $\mathcal B_0$ satisfies the required condition. Our main result is the following.


\begin{thm}
\label{main}
Suppose $\mathcal A$ is a set system on ground set $[n]$ such that $|A\backslash B|\neq 2|B\backslash A|$ for all distinct $A,B\in \mathcal A$. Then $|\mathcal A|\leq (1+o(1))\binom {n}{\lfloor n/2 \rfloor }$. Furthermore, if $n$ is even and sufficiently large then $|\mathcal A|\leq |\mathcal B_0|$, with equality 
if and only if $\mathcal A=\mathcal B_0$.
\end{thm}

\noindent The main step in the proof of Theorem \ref{main} is given by the following lemma. The proof is a Katona-type (see \cite{katona}) averaging argument.


\begin{lem}
\label{inequality}
Let $\mathcal{A}$ be a set system on $[n]$ such that $|A\backslash B|\neq 2|B\backslash A|$ for all distinct $A,B\in \mathcal A$. Then 
\begin{equation*}
 \sum _{j=l}^{2l} \frac {|\mathcal {A}_j|}{\binom {n}{j}} \leq 1
\end{equation*}
for all $l\leq \frac {n}{3}$ and 
\begin{equation*}
\sum _{j=2k-n}^{k} \frac {|\mathcal {A}_j|}{\binom {n}{j}} \leq 1 
\end{equation*}
for all $k\geq \frac {2n}{3}$, where $\mathcal A_j=\mathcal A\cap [n]^{(j)}$.
\end{lem}

\begin{proof}
We only prove the first inequality, as the proof of the second is identical. Pick a random ordering of $[n]$ which we denote by $(a_1,a_2,\ldots ,a_{\lceil \frac {2n}{3} \rceil}, b_1,\ldots ,b_{\lfloor \frac {n}{3} \rfloor})$. Given this ordering, let $C_i=\{a_j:j\in[2i]\}\cup \{b_k:k\in [i+1,l]\} $ and let $\mathcal{C}=\{C_i:i\in [0,l]\}$. Consider the random variable $X=|\mathcal {A}\cap \mathcal {C}|$. Since each set $B\in [n]^{(i)}$ is equally likely to be $C_{i-l}$ we have $\mathbb {P}[B\in \mathcal {C}]= \frac {1}{\binom {n}{i}}$. Thus by linearity of expectation we have

\begin{equation}
 \label{ref}
 \mathbb{E}(X)=\sum_{i=l}^{2l}\frac {|\mathcal A_i|}{\binom {n}{i}}
\end{equation}

On the other hand, given any $C_i, C_j$ with $i<j$ we have $|C_i\backslash C_j|=2|C_j\backslash C_i|$ and so $\mathcal{A}$ can contain at most one of these sets. This gives $\mathbb{E}(X)\leq 1$. Together with (\ref{ref}) this gives the claimed inequality

\begin{equation*}
 \sum_{i=l}^{2l}\frac {|A_i|}{\binom {n}{i}} \leq 1
\end{equation*}

\end{proof}


\noindent \emph{Proof of Theorem \ref{main}.} We first show $|\mathcal A| \leq (1+o(1))\binom {n}{\lfloor n/2 \rfloor}$. By standard estimates (See e.g. Appendix A of \cite{aands}) we have $|[n]^{(\leq \alpha n)} \cup [n]^{(\geq (1-\alpha)n)}| = o(\binom {n}{\lfloor n/2 \rfloor})$ for any fixed $\alpha \in [0,\frac {1}{2})$, so it suffices to show that $| \bigcup _{i={\frac {2n}{5}}}^{\frac {3n}{5}} \mathcal A_i|\leq \binom {n}{\frac {n}{2}}$. But this follows immediately from Lemma \ref{inequality} by taking $l=\lfloor \frac {n}{3} \rfloor $.

We now prove the extremal part of the claim in Theorem \ref{main}. We first show that the maximum of $f(x)=\sum _{i=0}^n x_i$ subject to the inequalities 
\begin{equation}
 \label{firstineq}
\sum _{j=l}^{2l} \frac {x_j}{\binom {n}{j}} \leq 1, \quad l\in \{ 0,1,\ldots ,\lfloor \frac {n}{3} \rfloor \}
\end{equation}
and 
\begin{equation}
 \label{secondineq}
\sum _{j=2k-n}^{k} \frac {x_j}{\binom {n}{j}} \leq 1, \quad k\in \{ \lceil \frac {2n}{3} \rceil ,\ldots ,n\}
\end{equation} 
from Lemma \ref{inequality} occurs when $x_{n/2}= \binom {n}{\frac {n}{2}}$. 
Indeed, suppose otherwise. At least one of these inequalities involving $x_{n/2}$ must occur with equality, as otherwise we can increase $x_{n/2}$ slightly, increase the value of $f(x)$ and still satisfy (\ref{firstineq}) and (\ref{secondineq}). 
Pick $j>\frac {n}{2}$ as small as possible such that $x_j>0$. Let 
$y_{n/2}=x_{n/2}+\epsilon \binom {n}{n/2}$, $y_j=x_j-\epsilon \binom {n}{j}$ and $y_i=x_i$ for all other $i$. As $f(y)>f(x)$ one of the (\ref{firstineq}) or (\ref{secondineq}) must fail. If $\epsilon$ is sufficiently small only the inequalities involving $y_{n/2}$ and not $y_j$ can be violated. Choose 
$k<n/2$ maximal such that $y_k>0$ and $y_k$ does not occur in any inequality involving $y_j$. Note that we must have $j-k\geq \frac {n}{4}$. Decrease $y_k$ by $\epsilon \binom {n}{k}$. Since the only increased variable $y_{n/2}$ always occurs with one of $y_j$ or $y_k$, it follows that $y=(y_0,\ldots ,y_n)$ satisfies (\ref{firstineq}) and (\ref{secondineq}). 

We claim that $f(y)>f(x)$. Indeed, we must have either $|j-\frac {n}{2}|\geq \frac {n}{8}$ or $|k-\frac{n}{2}|\geq \frac {n}{8}$. Without loss of generality assume that $|k-\frac {n}{2}|\geq \frac {n}{8}$. Then since 
$\binom {n}{n/2}> \binom {n}{(n/2)+1} + \binom {n}{3n/8}$ for sufficiently large $n$ we have 
\begin{equation*}
f(y)=f(x)+\epsilon \binom {n}{n/2}-\epsilon \binom {n}{j} - \epsilon \binom {n}{k}
>f(x)+\epsilon \binom {n}{n/2}-\epsilon \binom {n}{(n/2)+1} - \epsilon \binom {n}{3n/8}>f(x).
\end{equation*}
Therefore we must have $x_{n/2}=\binom {n}{n/2}$, as claimed.

Now, by the inequalities (\ref{firstineq}) and (\ref{secondineq})  we have 
$x_j=0$ for all $\frac {n}{4}\leq j\leq \frac {3n}{4}$ with $j\neq \frac {n}{2}$. From here it is easy to see by a weight transfer argument that $f(x)$ has a unique maximum when $x_i=\binom {n}{i}$ for $i\in I$ and $x_i=0$ otherwise. For a set system $\mathcal A$ these values of $x_i=|\mathcal A_i|$ can only be achieved if $\mathcal A=\mathcal B_0$, as claimed. \hspace{2cm} $\square$\\


\noindent We remark that the statement of Theorem \ref{main} does not hold for all even $n$, as can be seen for example by taking $n=4$ and $\mathcal A= \mathcal P[n]\backslash [n]^{(2)}$.


We now extend Theorem \ref{main} from the ratio $1:2$ to any given ratio $p:q$. Let $p:q$ be in its lowest terms and $p<q$. If $A\in [n]^{(i+a)}$ and $B\in [n]^{(i)}$ satisfy $p|A\backslash B|=q|B\backslash A|$ then we have $p(a+b)=q(b)$ where $b=|B\backslash A|$. But then $pa=(q-p)b$ and since $p$ and $q$ are coprime we must have that $(q-p)|a$. Therefore any family $\mathcal A=\bigcup _{i\in I}[n]^{(i)}$, where $I$ is an interval of length $q-p$, satisfies $p|A\backslash B|\neq q|B\backslash A|$ for all $A,B\in \mathcal A$. Taking $\lfloor \frac {n}{2}\rfloor \in I$ gives $|\mathcal A|=(q-p+o(1))\binom {n}{\lfloor n/2 \rfloor }$. Our next result shows that this is asymptotically best possible.

\begin{thm}
\label{givenratio}
Let $p,q\in \mathbb{N}$ be coprime with $p<q$. Let $\mathcal A$ be a set system on ground set $[n]$ such that $p|A\backslash B|\neq q|B\backslash A|$ for 
all distinct $A,B\in \mathcal A$. Then $|\mathcal A|\leq (q-p+o(1))\binom {n}{\lfloor n/2 \rfloor }$. 
\end{thm}

The following lemma performs an analogous role to that of Lemma \ref{inequality} in the proof of Theorem \ref{main}. 

\begin{lem}
\label{secondinequality}
Let $\mathcal{A}$ be a set system on $[n]$ such that $p|A\backslash B|\neq q|B\backslash A|$ for all distinct $A,B\in \mathcal A$. Then 
\begin{equation*}
 \sum _{j\in J_k} \frac {|\mathcal {A}_j|}{\binom {n}{j}} \leq 1
\end{equation*}
where $J_k=\{l:\lceil \frac {pn}{p+q}\rceil \leq l \leq \lfloor \frac {qn}{p+q}\rfloor, l\equiv k \pmod {(q-p)}\} $ for $0\leq k\leq q-p-1$.
\end{lem}

\begin{proof}
We only sketch the proof, as it is very similar to the proof of Lemma \ref{inequality}. For convenience we assume $n=(p+q)m$ (this assumption is easily removed). Fix $k\in [0,q-p-1]$ and let $k'\equiv k-pm\pmod {(q-p)}$ where $k'\in [0,q-p-1]$. Pick a random ordering of $[n]$ which we denote by $(a_1,a_2,\ldots ,a_{qm}, b_1,\ldots ,b_{pm})$. Given this ordering let $C_i=\{a_j:j\in[qi+k']\}\cup \{b_j:j\in [pi+1,pm]\} $ and let $\mathcal{C}=\{C_i:i\in [0,m-1]\}$. (Here if $k'=0$ we additionally adjoin $C_m$ to $\mathcal C$.) By choice of $k'$, we have $|C_i|\in J_k$ for all $i\in [0,m-1]$.

Again for any $C_i$ and $C_j$ with $i<j$ we have $q|C_i\backslash C_j|=p|C_j\backslash C_i|$, which implies that $\mathcal A$ contains at most one element of $\mathcal C$. Using this the rest of the proof is as in Lemma \ref{inequality}.
\end{proof}

The proof of Theorem \ref{givenratio} is now identical to the proof of Theorem \ref{main} taking Lemma \ref{secondinequality} in place of Lemma \ref{inequality}.

For simplicity we have given in Lemma \ref{secondinequality} only the inequalities that we needed in order to prove Theorem \ref{givenratio}. Further inequalities involving smaller level sets analogous to those in Lemma \ref{inequality} can also be obtained in a similar fashion. While we have not done so here, we note that it is possible to use these inequalities to again find an exact extremal family for any given ratio $p:q$ as in Theorem \ref{main}, provided $q-p$ and $n$ have the opposite parity and $n$ is sufficiently large.


\section{Forbidding a fixed distance}

In this final section we consider how large a family $\mathcal A$ can be if for all $A,B\in \mathcal A$ we do not allow $A$ to have a constant distance from the bottom of the subcube formed with $B$. For `distance exactly 1' this would mean that we exclude $\vert A \backslash B\vert= 1$ for $A,B\in \mathcal A$. Here the following family $\mathcal A^*$ provides a lower bound: let $\mathcal A^*$ consist of all sets $A$ of size $\lfloor n/2\rfloor $ such that $\sum _{i\in A} i \equiv r \pmod {n}$, where $r\in \{0,\ldots ,n-1\}$ is chosen to maximise $|\mathcal A^*|$. Such a choice of $r$ gives $| \mathcal A^* |\geq {\frac {1}{n}} {\binom{n} {\lfloor n/2\rfloor }}$. Note that if we had $|A\backslash B|=1$ for some $A,B\in \mathcal A^*$ then, since $|A|=|B|$, we would also have $|B\backslash A|=1$. Letting $A\backslash B=\{i\}$ and $B\backslash A=\{j\}$ we then have $i-j\equiv 0 \pmod {n}$, giving $i=j$, a contradiction.


We suspect that this bound is best.
\begin{conjecture} 
\label{conject}
Let $\mathcal{A}\subset \mathcal{P}[n]$ be a family which satisfies $|A\backslash B|\neq 1$ for all $A,B\in \mathcal{A}$. Then $|\mathcal A|\leq (1+o(1))\frac {1}{n}\binom {n}{\lfloor n/2\rfloor }$.
 \end{conjecture}

\noindent The following gives an upper bound that is a factor $n^{1/2}$ larger than this.


\begin{thm}
{
\label{exact}
Let $\mathcal{A}\subset \mathcal{P}[n]$ be a family such that $|A\backslash B|\neq 1$ for all $A,B\in \mathcal{A}$. Then there exists a constant $C$ independent of $n$ such that $|\mathcal{A}|\leq \frac{C}{n}2^n$.}
\end{thm}

\begin{proof}
 {
An easy estimate gives that the number of subsets of $\mathcal{A}$ in 
$[n]^{(\leq n/3)}\bigcup [n]^{(\geq 2n/3)}$ is at most 
$4\binom{n}{n/3}=o(\frac{2^n}{n})$. Therefore it suffices to show that $|\mathcal{A}_i|\leq \frac{C}{n}\binom{n}{i}$ for all $i\in [\frac{n}{3},\frac{2n}{3}]$. 

To see this, note that since $|A\backslash A'|\neq 1$ for all $A,A'\in \mathcal{A}$, each $B\in [n]^{(i+1)}$ contains at most one $A\in \mathcal{A}_i$. Double counting, we have 

\begin{equation*}
{
\begin{split}
\frac {n}{3} | {\mathcal {A}} _i| \leq (n-i)| {\mathcal {A} }_i| 	&= | \lbrace (A,B): A\in \mathcal {A}_i, B \in [n]^{(i+1)}, A\subset B\rbrace | \\
									& \leq \binom {n}{i+1} \leq 3\binom {n}{i}
\end{split}
}
\end{equation*}
as required.
}
\end{proof}


Our final result gives an upper bound on the size of a family $\mathcal A$ in which we forbid `distance at most 1' instead of `distance exactly 1', i.e. 
where we have $|A\backslash B|> 1$ for all $A,B\in \mathcal A$. Again, the family $\mathcal A^*$ constructed above gives a lower bound for this problem. In general, if we forbid `distance at most $k$' then it is easily seen that the following family $\mathcal A_k^*$ gives a lower bound of $\frac {1}{n^k}\binom {n}{\lfloor n/2\rfloor }$: supposing $n$ is prime, let $\mathcal A_k^*$ consist of all sets $A$ of $\lfloor n/2 \rfloor $ which satisfy $\sum _{i\in A}i^d\equiv 0\pmod {n}$ for all $1\leq d\leq k$. 

Our last result provides a upper bound which matches this up to a multiplicative constant. The proof is again a Katona-type argument. Here the condition $|A\backslash B|>k$ rather than $|A\backslash B|\neq k$ seems to be crucial.


\begin{thm}
\label{atmost}
Let $k\in \mathbb {N}$. Suppose $\mathcal {A}$ is a set system on $[n]$ such that $|A\backslash B|>k$ for all distinct $A,B \in \mathcal {A}$. Then $|\mathcal {A}|\leq \frac {(2^k-o(1))}{n^k}\binom {n}{\lfloor n/2 \rfloor}$.
\end{thm}

\begin{proof}
{
Consider the family $\partial ^{(k)} \mathcal A$, the $k$-shadow of $\mathcal A$, where 
\begin{equation*}
{
\partial ^{(k)}\mathcal{A} = \{B\in \mathcal{P}[n]: B=A\backslash C 
\mbox{ for some } A\in \mathcal {A} \mbox{ and }C\subset A \mbox { with } |C|=k\}.
}
\end{equation*}

Since $\mathcal{A}$ does not contain $A,B$ with $|A\backslash B|\leq k$, every element of $\partial ^{(k)}\mathcal{A}$ is contained in at most one element of $\mathcal{A}$. Therefore we have

\begin{equation}
{
\label{firstref}
|\partial ^{(k)}\mathcal{A}|=\sum_{i=0}^n (i)_k|\mathcal{A}_i|
}
\end{equation} 
where $i_k=i(i-1)\cdots (i-k+1)$. Now, since $\mathcal{A}$ does not contain $A,B$ with $|A\backslash B|\leq k$, it follows that $\partial ^{(k)}\mathcal{A}$ is an antichain, and so by Sperner's theorem we have 
\begin{equation}
 {
\label{secondref}
|\partial ^{(k)}\mathcal{A}| \leq \binom{n}{\lfloor n/2 \rfloor}
}
\end{equation}
Finally, an estimate of the sum of binomial coefficients (Appendix A of \cite{aands}) gives 
\begin{equation}
 {
\label{thirdref}
\sum_{i=0}^{\frac{n}{2}-n^{2/3}}|\mathcal{A}_i| \leq \sum_{i=0}^{\frac{n}{2}-n^{2/3}} \binom{n}{i} \leq e^{-n^{1/3}}2^n.
}
\end{equation}
Combining (\ref{firstref}), (\ref{secondref}) and (\ref{thirdref}) we obtain
\begin{equation*}
{
\begin{split}
\binom {n}{\lfloor n/2 \rfloor}  &\geq \sum_{i=0}^{\frac{n}{2}-n^{2/3}} (i)_k |\mathcal {A}_i| + \sum_{i=\frac{n}{2}-n^{2/3}}^n (i)_k|\mathcal {A}_i| \\ 
			  &\geq  \sum_{i=0}^{\frac{n}{2}-n^{2/3}} (\frac {n}{2}-n^{2/3})_k|\mathcal {A}_i| - (\frac {n}{2}-n^{2/3})_ke^{-n^{1/3}}2^n + \sum_{i={\frac{n}{2}-n^{2/3}}}^n (\frac{n}{2}-n^{2/3})_k|\mathcal {A}_i| \\
			  &=  (\frac {n}{2}- o(n))^k|\mathcal {A}| - o( \binom {n}{\lfloor n/2 \rfloor} )
\end{split}
}
\end{equation*}
which gives the desired result. 
}
\end{proof}


Taking $k=1$ in Theorem \ref{atmost} we obtain an upper bound which differs by a factor of 2 from the lower bound given by the family $\mathcal A^*$. It would be interesting to close this gap.


\end{document}